\documentclass[a4paper,11pt,finnish]{article}
\usepackage{paper3_information_spreading}

\usepackage{hyperref}
\hypersetup{
    colorlinks=true,
    linkcolor=blue,
    filecolor=magenta,      
    urlcolor=cyan,
}
\usepackage{booktabs}
\usepackage{etoolbox}

\begin{document}
\title{Flooding in weighted sparse random graphs of active and passive nodes}
\author{Hoa Ngo}
\maketitle

\setcounter{page}{1}
\begin{abstract}
This paper discusses first passage percolation and flooding on large weighted sparse random graphs with two types of nodes: active and passive nodes.
In mathematical physics passive nodes can be interpreted as closed gates where fluid flow or water cannot pass through and active nodes can be interpreted as open gates where water may keep flowing further. The model of this paper has  many applications in real life, for example, information spreading, where passive nodes are interpreted as passive receivers who may read messages but do not respond to them. In the epidemic context passive nodes may be interpreted as individuals who self-isolate themselves after having a disease to stop spreading the disease any further. When all weights on edges between active nodes and between active and passive nodes are independent and exponentially distributed (but not necessary identically distributed), this article provides an approximation formula for the weighted typical flooding time.
\end{abstract}

{\bf Keywords}: first passage percolation, weighted flooding time, general configuration model, sparse random graph, active and passive node

\section{Introduction}
First passage percolation is one of the classical models in probability theory and mathematical physics introduced by Hammersley and Welsh \cite{HW} in $1965$ as a generalization of Bernoulli percolation \cite{BH,DC}. Its large interest due to the model simplicity and its various applications from theoretical physics to biology \cite{ME, GK, AP, GM, OG}. 
The first study rooted in $1957$ by Broadbent and Hammersley \cite{BH}, when they studied  
the first passage time of fluid flow through the random properties of a porous medium.
The porous medium can be modelled as a random undirected connected graph $G$ where each undirected edge $e$ is attached with an independent nonnegative random weight $W_e$ drawn from distribution $L_e$. The weight $W_e$ represents a transmission time of fluid flow from an active node to another node.
Hammersley and Welsh \cite{HW} defined a transmission time from $a$ to $b$ along a connected path $\pi$ on $G$ by 
\begin{align}
\label{definition: The active transmitter time of the path}
t_G(\pi)= \sum_{e\in\pi}W_e.
\end{align}
The first passage time from $u$ to $v$ on $G$ is defined by
\begin{align*}
\tau_G(u,v) = \inf_{\pi\in S}t_G(\pi),
\end{align*}
where $S$ is a set of all possible paths from $a$ to $b$. 
The flooding time of node $a$ in $G$ is defined by 
\begin{align}
\flood_G(a) = \max_b\tau_G(a,b),
\end{align}
where the maximum is taken over all nodes $b$ in $G$. \\

When $G$ is complete graph with $n$ nodes and each edge weight is independent and has an identical exponential distribution $L_e = \Exp(1)$, Janson \cite{J} showed in $1998$ that
\begin{align}
\frac{\tau_G(a,b)}{\log n/n}&\prto 1,\\
\frac{\flood_G(U_n)}{\log n/n}&\prto 2,\\
\frac{\max_a\flood_G(a)}{\log n/n}&\prto 3,
\end{align}
where $U_n$ is a uniformly chosen node in $[n]=\{1,2,3,...,n\}$ and notation "$\prto$" means convergence in probability.
The results have inspired many studies following other results in various settings (see e,g. \cite{DKLP,BHH, BHH17, ADL, AL, PAL, LN19}).
When $G = G(n, (d_i)_{i=1}^{n})$ is a sparse random graph with a given degree sequence $(d_i)_{i=1}^{n}$  constructed via configuration model \cite{B,RHI} and edge weights are independently of each other and identically exponentially distributed with parameter $1$ ($L_e = \Exp(1)$), it has been showed in \cite{ADL} that,
\begin{align}
\frac{\tau_G(a,b)}{\log n}&\prto \frac{1}{\nu-1},\\
\frac{\flood_G(U_n)}{\log n}&\prto \frac{1}{\nu-1}+\frac{1}{\dmin},\\
\frac{\max_a\flood_G(a)}{\log n}&\prto \frac{1}{\nu-1}+\frac{2}{\dmin},
\end{align}
where $\dmin = \min_{1\le i\le n} d_i\ge 3$, 
$\nu = \lim_{n\to\infty}\frac{\E(d_{U_n}^2)}{\E(d_{U_n})}-1\gr 1$ and $U_n$ is a uniformly chosen node in $[n]$ .\\

This paper generalizes the model of Hammersley and Welsh by including passive nodes and studies how these nodes slow down the typical flooding time on a large weighted sparse random graph. 


\section{Definitions and notations} 

Let $G = (V,E)$ be a finite undirected random graph, where $V$ is the set of nodes and $E$ is the set of edges. In this paper it is assumed that the set of nodes $V$ consists of active nodes $V_1$ and passive nodes $V_2$; the set of edges $E$ consists of edges from active to active $E_{11}$, from active to passive $E_{12}$ and from passive to passive $E_{22}$. The indices $i=1,2$ are used to refer to different types of nodes: active and passive nodes. The size of set $V$ is denoted by $n := \abs{V}$. Note that $V = V_1\cup V_2$ and $n = n_1 + n_2$, where $n_1$ and $n_2$ are the number of active and passive nodes, respectively. By symmetry, we see that $E_{12} = E_{21}$, and hence $E = E_{11}\cup E_{12}\cup E_{22}$.


\subsubsection*{Walkable path}

A path of length $\ell$ from $a$ to $b$ is a sequence $\pi:= (v_0,e_1,v_1,e_2,\dots, v_{\ell-1},e_\ell,v_\ell)$, where nodes $v_0 = a, v_1,\dots, v_\ell=b$  in $V$ (not necessary distinct) and edges between two nodes $e_j=\{v_{j-1},v_j\}$  ($j=1,2,\dots,\ell$) in $E$. A path $\pi$ from $a$ to $b$ is said to be \textit{walkable} if nodes $(v_j)_{j\ls\ell}$ in $\pi$ are active.
An inverse path of $\pi$ is defined by $\pi^{-1}:= (v_\ell,e_{\ell},v_{\ell-1},\dots,e_2,v_1,e_1,v_0)$.
We say that a path $\pi$ is \textit{strongly walkable}, if also its inverse path is walkable. Note that all paths in undirected (active) subgraph $G_1: = (V_1,E_{11})$ are strongly walkable.

\subsection{Weighted random graphs}
An undirected graph $G$ equipped with random weights $W = \{W_e\}$ on its edges is called a \textit{weighted random undirected graph} $G$.
\subsubsection*{Weighted first passage time} 
The \textit{weighted first passage time} from $a$ to $b$ on $G$ is defined by
\begin{align*}
\tau(a,b) = \inf_{\pi\in S} \sum_{e\in\pi} W_e,
\end{align*}
where $S$ is the set of all walkable paths from $a$ to $b$. If there does not exist a walkable path from $a$ to $b$ (\ie $S = \emptyset$), we set $\tau(a,b) = \infty$.

\subsubsection*{Weighted flooding times} 
The weighted flooding time of active node $a$ in $V_1$ and on $G$ is defined by
\begin{align*}
\label{definition:flooding time}
\flood(a) &= \max\{\tau(a,b)\ls \infty : b\in V\}.
\end{align*}

\subsection{General configuration model}
\label{configuration model of three types of half-edges}
A classical way to construct an undirected random graph $G$ with a given degree sequence $(d_i)_{i=1}^{n}$  is to use configuration model introduced by \Bollobas \cite{B} (see also e.g.\; \cite{RHI}). This classical configuration consists only one type of nodes and half-edges, where each half-edge is matched uniformly at random to unmatched half-edges until none of them are left and the same processes is repeated until all half-edges are matched. Let us introduce in the following a \textit{general configuration model with two types of nodes} (see related work \cite{CO,SB}). In this presented general configuration model type-$1$ and type-$2$ nodes can be anything, need not be necessary active or passive nodes. \\

Let $\bd = \{(d_{11}(v))_{v\in V_1}, (d_{12}(v))_{v\in V_1}, (d_{21}(v))_{v\in V_2}, (d_{22}(v))_{v\in V_2}\}$ be a collection of integer sequences satisfying the following properties:
\begin{itemize}
\item[(i)] $\sum_{v\in V_1}d_{11}(v)$ and $\sum_{v\in V_1}d_{22}(v)$ are even,
\item[(ii)]
\label{balance condition}
$\sum_{v\in V_1}d_{12}(v) = \sum_{v\in V_2}d_{21}(v)$ (\textit{the sum condition of bipartite graph}).
\end{itemize}
For each node $u$ in $V_1$ we attach $d_{11}(u)$ and $d_{12}(u)$ labelled elements called type-$11$ and type-$12$ half-edges, respectively. Similarly  each node $v$ in $V_2$ are attached with $d_{22}(v)$ and $d_{21}(v)$ labelled elements called type-$22$ and type-$21$ half-edges, respectively. We pair half-edges uniformly at random until no half-edges are left according to the following rules:
\begin{enumerate}
\label{matching 1}
\item Type-$ii$ half-edges are paired with each other for all $i=1,2$,
\item Each type-$12$ half-edge is paired to a type-$21$ half-edge.
\end{enumerate}
A pair of matched half-edges in above form the different types of edges: type-$ij$ edges, where $i,j = 1,2$. The matchings may result in self-loops or parallel edges, and hence the obtained graph is called a multigraph, denoted by  $\tG = (V,\bd)$, where $V=V_1\cup V_2$. Note that in this construction $\tG$ has three multi-subgraphs \ie $\tG = \tG_1\cup \tG_2\cup \tG_3$, where $\tG_1:=(V_1, \bd_{11})$,  $\tG_2:=(V_2, \bd_{22})$, $\tG_3:=(V, \bd_{12}\cup\bd_{21})$ (obtained by pairing half-edges according to the above rules $1$ and $2$) and $\bd_{ij} = (d_{ij}(v))_{v\in V_i}$ for all $i,j= 1,2$. 
\\ 

A graph is said to be \textit{simple} if it does not contain self-loops and parallel edges. Conditional on $\tG$ is being simple, we obtain a simple random graph with given degree sequence $\bd$ denoted by $G=G_{1}\cup G_{2}\cup G_{3}$ \cite{BS13,J09}, where all subgraphs $G_k$ ($k=1,2,3$) are simple and uniformly distributed with given degree sequences \cite{RHI}.  We see that $G$ is simple if and only if all its subgraphs $G_k$ are simple. To ensure the probability that $\tG$ is simple stays away from zero, we assume that the degree sequences $\bd_{11}$ and $\bd_{22}$ satisfy \Erdos--Gallai conditions \cite{MOA} and the degree sequence $\bd_{12}\cup\bd_{21}$ satisfies Gale--Ryser conditions \cite{G57}.


\section{Model descriptions}
\label{model}
Let $G^{(\kappa)} = (V^{(\kappa)},\bd^{(\kappa)})$ be a random graph of active and passive nodes constructed as in Chapter \ref{configuration model of three types of half-edges}, where each index $\kappa$ takes positive integer value. All edges in $G^{(\kappa)}$ are attached with random weights $W_e^{(\kappa)}$ according to the following:
\begin{itemize}
\item $W_e^{(\kappa)} \st \Exp(\lambda_{11})$ if $e\in E_{11}^{(\kappa)}$,
\item $W_e^{(\kappa)} \st \Exp(\lambda_{12})$ if $e\in E_{12}^{(\kappa)}$,
\item all edge weights mentioned above are independent, 
\end{itemize}
where rate parameters $\lambda_{11}$ and $\lambda_{12}$ are strictly positive.
\subsection{Conditions for nodes}
The number of active and passive nodes are assumed to have the same order of $\kappa$ i.e. $n_1^{(\kappa)}=\btheta(\kappa)$ and $n_2^{(\kappa)}=\btheta(\kappa)$ implying also that $n^{(\kappa)} =\btheta(\kappa)$. 
The empirical type-$i1$ degree distribution is  defined by
\[p_{i1}^{(\kappa)}(j) = \frac{\abs{\{v\in V_i^{(\kappa)}:d_{i1}^{(\kappa)}(v) = j\}}}{n_i^{(\kappa)}}\]
for all integers $j\ge 0$, where $i=1,2$. The degree sequence $\bd^{(\kappa)}$ is assumed to satisfies the following two conditions: Regularity conditions \ref{regularity conditions} (introduced in \cite{MR} or \cite{RHII}) and Blanchet and Stauffer conditions  \ref{asymptotic bipartite graph conditions} \cite{BS13} (asymptotic bipartite graph conditions).

\begin{condition} 
\label{regularity conditions}
\textit{(Regularity conditions)}. There exist limiting distributions $p_{i1} = (p_{i1}(j))_{j\in\N}$ ($i=1,2$) and positive integer $\kappa_0$ such that
\begin{enumerate}
\item[(1)] $p_{i1}^{(\kappa)}(j)\to p_{i1}(j)$ for all integers $j\ge 1$ as $\kappa\to\infty$,
\item[(2)] $\sum_j j^{2+\epsilon} p_{i1}(j) = O(1)$ for some $\epsilon\gr 0$ (bounded second moment),
\item[(3)] $\min_{v\in V_i}d_{i1}^{(\kappa)}(j) = \delta_{i1} $ for all $\kappa\ge \kappa_0$ and $p_{i1}(\delta_i)\gr 0$, where $\delta_{11}\ge 3$ and  $\delta_{21}\ge 1$.
\end{enumerate}
\end{condition}

The means of limiting distribution $p_{11}$ and its downshifted size biasing
distribution (see definition in \cite{LN17}) are denoted by
\begin{align*}
\mu_{11} := \sum_{j}j p_{11}(j)\quad\text{and}\quad \nu_{11} := \frac{1}{\mu_{11}}\sum_{j}j(j-1)p_{11}(j),
\end{align*} 
respectively, and both means are assumed to be strictly positive numbers.\\

Let us denote $N = \sum _{v\in V_1}d_{12}^{(\kappa)}(v) = \sum _{v\in V_2}d_{21}^{(\kappa)}(v)$. 
We order degree sequences $(d_{12}^{(\kappa)}(v))_{v\in V_1}$ and $(d_{21}^{(\kappa)}(v))_{v\in V_2}$ in decreasing orders as $s_1\ge s_2\ge\cdots\ge s_{n_1}$ and $t_1\ge t_2\ge\cdots\ge t_{n_2}$, respectively. Denote $s = \max_i s_i$ and $t = \max_i t_i$.

\begin{condition} (Blanchet and Stauffer conditions \cite{BS13}). 
\label{asymptotic bipartite graph conditions}
The following two conditions hold:
\begin{enumerate}
\item[(i)]\
\begin{align*}
\sum_i\sum_j s_i(s_i-1)t_j(t_j-1) = O(N^2).
\end{align*}
\item[(ii)]  For any $m\ge 1$,
\begin{align*}
\sum_{i= t \wedge m}^{n_1} s_i &= \Omega(N),\\
\sum_{i= s \wedge m}^{n_2} t_i &= \Omega(N).
\end{align*}
\end{enumerate}

\end{condition}

\subsection{Main result}
\label{main result}
\begin{theorem}
Consider a sequence of random simple graphs $G^{(\kappa)}=(V^{(\kappa)},\bd^{(\kappa)})$ so that its degree sequence $\bd^{(\kappa)}$ satisfies  Conditions \ref{regularity conditions} and \ref{asymptotic bipartite graph conditions}. Then for a uniformly chosen active node $A$ in $G^{(\kappa)}$,
\begin{align*}
\frac{\flood(A)}{\log\kappa} &\prto \frac{1}{\lambda_{11}(\nu_{11}-1)} + \frac{1}{\lambda_{11}\delta_{11}\wedge\lambda_{12}\delta_{21}}
\end{align*}
as $\kappa$ tends to infinity.
\end{theorem}

\section{Proof of main result}
To keep notations simple all indexes $\kappa$ are dropped out from now (unless otherwise mentioned), and it is written $n_1, n_2, n, V_1, d_{11}$ and so on. Note that $n_i\to\infty$ if and only if $\kappa\to\infty$.\\  

By assumption the minimum type-$11$ degree is at least $3$, and hence $G_1$ is connected \whp (see e.g.\;\cite{ADL}[Lemma 2.1]). Since each passive node has at least an active neighbour ($\delta_{21}\ge 1)$, we conclude that \whp $G$ is connected. Furthermore, for any given two nodes $a$ in $V_1$ and $b$ in $V$, \whp there exists a walkable path from $a$ to $b$ in $G$. Note that $G$ may remain still connected even if its subgraphs $G_2$ and $G_3$ are not connected. To ease an analysis of weighted flooding time we may remove all edges between two passive nodes since their edge weights do not affect on flooding time.

\subsection{Proof of upper bound}
In this section it will be showed that for any $\epsilon\gr 0$ \whp,
\begin{align}
\label{typical flooding time: upper bound}
\flood(A) &\le \Big(\frac{1}{\lambda_{11}(\nu_{11}-1)} + \frac{1}{\lambda_{11}\delta_{11}\wedge\lambda_{12}\delta_{21}}\ + \epsilon\Big)\log \kappa. 
\end{align}
We notice that the flooding time of $A$ consists of two parts: flooding from $A$ to all active nodes and flooding from $A$ to all passive nodes.  These  quantities are defined by
\begin{align*}
\flood_1(A) = \max_{b\in V_1}\tau(A,b)\quad\quad\text{and}\quad\quad \flood_2(A) = \max_{b\in V_2}\tau(A,b)
\end{align*}
and called the \textit{flooding time 1} and the \textit{flooding time 2}, respectively. In the following subchapters it will be derived upper bounds for flooding times $1$ and $2$.

\subsubsection{Upper bound of flooding time 1}

\subsubsection*{Definitions and notations}
The set of $t$-radius \textit{active neighbourhood} of $a$ in $V_1$ is defined by
\begin{align*}
B_1(a,t) = \{b\in V_1: \tau(a,b)\le t\}
\end{align*}
Note that the restriction on set $V_1$, $\tau$ is metric on $V_1$. The time to reach $k$ active nodes is defined by 
\begin{align*}
T_a(k) = \min\{t\ge 0: \abs{B_1(a,t)}\ge k+1\}
\end{align*}
Denote scale parameters (depending on $\kappa$)
\begin{align*}
 \alpha &= \lfloor \log^3 n_1 \rfloor,\\
 \beta &=\Big\lfloor 3\sqrt{\tfrac{\mu_{11}}{\nu_{11}-1}n_1\log n_1} \Big\rfloor,
\end{align*}
and define $T_A(\alpha,\beta) = T_A(\beta) -T_A(\alpha)$.
Let us introduce the next three results from \textit{Amini and Lelarge} \cite{AL} (or see alternatively \cite{DKLP}).

\begin{proposition}
\label{weighted flooding time 1: two balls with the size of beta intersect}
With high probability,
\begin{align*}
\tau(a,b) \le T_a(\beta) + T_b(\beta)
\end{align*}
for all $a$ and $b$ in $V_1$. 
\end{proposition}

\begin{lemma}
\label{weighted flooding time 1: typical time to reach alpha active nodes}
For a uniformly chosen active node $A$ in $V_1$ and any $\epsilon\gr 0$,
\begin{align*}
\pr\Big(T_A(\alpha) \ge \Big(\frac{1}{\lambda_{11}\delta_{11}} + \epsilon\Big)\log n_1\Big) = o(n_1^{-1}).
\end{align*}
\end{lemma}
\begin{lemma}
\label{weighted flooding time 1: typical time to reach from alpha to beta active nodes}
For a uniformly chosen active node $A$ in $V_1$ and any $\epsilon\gr 0$,
\begin{align*}
\pr\Big(T_A(\alpha,\beta) \ge \Big(\frac{1}{\lambda_{11}(\nu_{11}-1)}+\epsilon\Big)\log n_1\Big) = o(n_1^{-1}).
\end{align*}
\end{lemma}
Applying the above three results, we have the following results.
\begin{proposition} For any $\epsilon\gr 0$ \whp,
\label{weighted flooding time 1: time to reach beta active nodes}
\begin{enumerate}
\item [(1)]$T_A(\beta) \le \frac{1}{2\lambda_{11}(\nu_{11}-1)}\log n_1 + \epsilon\log n_1$,
\item [(2)]$T_a(\beta) \le \frac{1}{2\lambda_{11}(\nu_{11}-1)}\log n_1 + \frac{1}{\lambda_{11}\delta_{11}}\log n_1  + \epsilon\log n_1$\:\: for all $a\in V_1$.
\end{enumerate}
\end{proposition}
\begin{proof} 
\begin{enumerate}
\item[(1)] By Lemmas \ref{weighted flooding time 1: typical time to reach alpha active nodes} and \ref{weighted flooding time 1: typical time to reach from alpha to beta active nodes} \whp,
\begin{align*}
T_A(\beta) &= T_A(\alpha) + T_A(\alpha,\beta)\\
&\le  \frac{1}{2\lambda_{11}(\nu_{11}-1)}\log n_1 + 2\epsilon\log n_1.
\end{align*}
\item[(2)] By Lemmas \ref{weighted flooding time 1: typical time to reach alpha active nodes} and \ref{weighted flooding time 1: typical time to reach from alpha to beta active nodes} \whp,
\begin{align*}
\max_a T_a(\beta) &= \max_a \big(T_a(\alpha) + T_a(\alpha,\beta)\big)\\
&\le\frac{1}{2\lambda_{11}(\nu_{11}-1)}\log n_1 + \epsilon\log n_1 + \max_a T_a(\alpha)\\
&\le\frac{1}{2\lambda_{11}(\nu_{11}-1)}\log n_1 + \frac{1}{\lambda_{11}\delta_{11}}\log n_1 + 2\epsilon\log n_1.\\
\end{align*} 
\end{enumerate}
\end{proof}

\begin{theorem} For any $\epsilon\gr 0$ \whp,
\label{weighted flooding time 1:upper bound}
\begin{align*}
\flood_1(A)\le\frac{1}{\lambda_{11}(\nu_{11}-1)}\log \kappa + \frac{1}{\lambda_{11}\delta_{11}}\log \kappa + \epsilon\log \kappa.
\end{align*}
\end{theorem}

\begin{proof}
By Propositions \ref{weighted flooding time 1: two balls with the size of beta intersect} and \ref{weighted flooding time 1: time to reach beta active nodes} \whp,
\begin{align*}
\flood_1(A) &= \max_u\tau(A,u)\\
&\le \max_u \big(T_A(\beta) + T_{u}(\beta)\big)\\
&\le T_A(\beta) + \max_u T_{u}(\beta)\\
&\le\frac{1}{\lambda_{11}(\nu_{11}-1)}\log n_1 + \frac{1}{\lambda_{11}\delta_{11}}\log n_1 + 2\epsilon\log n_1.
\end{align*}
The claim follows now by using assumption $n_1(\kappa) = \Theta(\kappa)$.
\end{proof}

\subsubsection{Upper bound of flooding time 2}
We define the active neighbourhood and the number of active neighbours of $v$ in $V$ by
\[
N_1(v) = \{a\in V_1: \{a,v\} \in E\}\quad\quad\text{and}\quad\quad \deg_1(v) = \abs{N_1(v)},
\]
respectively. The quantity $\deg_1(v)$ is called type-$1$ degree of $v$ interpreted as the number of active neighbours of node $v$. Note that in this paper $\deg_1(v)$ is not random, since the degree sequences $(d_{11}(v))_{v\in V_1}$ and $(d_{21}(v))_{v\in V_1}$ are fixed.\\

Recall that every passive node $b$ is assumed to have at least one active neighbour ($\delta_{21}\ge 1)$. Hence, for any passive node $b$ in $V_2$ there exists an active neighbour of $b$, denoted by  $u_b$, that has a minimum edge weight $W_b = \min_{a\in N_1(b)}W_{ab}$ between $b$ and $u_b$. Let us introduce the following proposition.

\begin{proposition}
\label{weighted flooding time 2; upper bound: two balls with the size of beta nodes intersect}
With high probability,
\begin{align}
\tau(a,b) \le T_a (\beta) + T_{u_b}(\beta) + W_b
\end{align}
for all $a$ in $V_1$ and $b$ in $V_2$. 
\end{proposition}

\begin{proof}
By the definition of the first passage time and Proposition \ref{weighted flooding time 1: two balls with the size of beta intersect} \whp,
\begin{align*}
\tau(a,b)&\le \tau(a,u_b) + W_b\\
&\le T_a(\beta) + T_{u_b}(\beta) + W_b
\end{align*}
for all $a\in V_1$ and $b\in V_2$.
\end{proof}

\begin{proposition} 
\label{weighted flooding time 2: time to reach alpha + 1 active nodes}
For a uniformly chosen passive node $B$ in $V_2$, any $0\le x\le 1$ and any $\epsilon\gr 0$,
\begin{align*}
\pr\left( T_{u_B}(\alpha_{n_1}) + W_B \ge \Big( \frac{x}{\lambda_{11} \delta_{11} \wedge \lambda_{12}\delta_{21}} + \epsilon \Big) \log n_1 \right)
 = o(n_1^{-x}).
\end{align*}
\begin{proof}
Let $G_{\delta_{11}}$ be $\delta_{11}$-regular random graph on $n_1$ nodes (see definition \cite{DKLP}). Let $u^*$ be a uniformly chosen node and $T_{u^*}(\alpha_{n_1})$ be the time to reach $\alpha_{n_1}$ nodes in $G_{\delta_{11}}$. Denote $t_{n_1} = \Big(\frac{x}{\lambda_{11} \delta_{11} \wedge \lambda_{12}\delta_{21}} + \epsilon \Big) \log n_1$ and let  $X$ be an exponential random weight with rate parameter $\lambda_{11}\delta_{11}\gr 0$ independent of $T_{u^*}(\alpha_{n_1})$ and $G_{\delta_{11}}$. By stochastic ordering $T_{u_B}(\alpha_{n_1})\lest T_{u^*}(\alpha_{n_1})$ and $W_B\lest X$, and independences of random variables,
\begin{align*}
T_{u_B}(\alpha_{n_1}) + W_B \lest T_{u^*}(\alpha_{n_1}) + X.
\end{align*}
Now by \cite{LN19}[Proposition 1] $\pr(T_{u^*}(\alpha_{n_1}) + X \ge t_{n_1}) = o(n_1^{-x})$, and hence,
\begin{align*}
\pr( T_{u_B}(\alpha_{n_1}) + W_B \ge t_{n_1})\le \pr(T_{u^*}(\alpha_{n_1}) + X \ge t_{n_1}) = o(n_1^{-x}).
\end{align*}
\end{proof}
\end{proposition}

\begin{proposition} For any $\epsilon\gr 0$ \whp,
\label{weighted flooding time 2: time to reach beta active nodes}
\begin{enumerate}
\item [(i)]$T_{u_B}(\beta) + W_B \le \frac{1}{2\lambda_{11}(\nu_{11}-1)}\log n_1 + \epsilon\log n_1$,
\item [(ii)]$T_{u_b}(\beta) + W_b \le \frac{1}{2\lambda_{11}(\nu_{11}-1)}\log n_1 + \frac{1}{\lambda_{11}\delta_{11}\wedge\lambda_{12}\delta_{21}}\log n_1  + \epsilon\log n_1$\\ for all $b\in V_2$.
\end{enumerate}
\end{proposition}
\begin{proof} By Proposition \ref{weighted flooding time 2: time to reach alpha + 1 active nodes}, Lemma \ref{weighted flooding time 1: typical time to reach from alpha to beta active nodes} and the union bound \whp,
\begin{enumerate}
\item[(i)] 
\begin{align*}
T_{u_B}(\beta) + W_B & = T_{u_B}(\alpha) + W_B + T_{u_B}(\alpha,\beta)\\
&\le \frac{1}{2\lambda_{11}(\nu_{11}-1)}\log n_1 + 2\epsilon\log n_1.
\end{align*}
\item[(ii)]
\begin{align*}
\max_b (T_{u_b}(\beta) + W_b) &= \max_b \big(T_{u_b}(\alpha) + W_b + T_{u_b}(\alpha,\beta)\big)\\
&\le\frac{1}{2\lambda_{11}(\nu_{11}-1)}\log n_1 + \epsilon\log n_1 + \max_b\ \big(T_{u_b}(\alpha) + W_b\big)\\
&\le\frac{1}{2\lambda_{11}(\nu_{11}-1)}\log n_1 + \frac{1}{\lambda_{11}\delta_{11}\wedge\lambda_{12}\delta_{21}}\log n_1 + 2\epsilon\log n_1.\\
\end{align*} 
\end{enumerate}
\end{proof}

\begin{theorem} For any $\epsilon\gr 0$ \whp,
\label{weighted flooding time 2:upper bound}
\begin{align*}
\flood_2(A)\le\frac{1}{\lambda_{11}(\nu_{11}-1)}\log \kappa + \frac{1}{\lambda_{11}\delta_{11}\wedge\lambda_{12}\delta_{21}}\log \kappa + \epsilon\log \kappa.
\end{align*}
\end{theorem}

\begin{proof}
By Propositions \ref{weighted flooding time 1: time to reach beta active nodes}, \ref{weighted flooding time 2; upper bound: two balls with the size of beta nodes intersect} and \ref{weighted flooding time 2: time to reach beta active nodes} \whp,
\begin{align*}
\flood_2(A) &= \max_b\tau(A,b)\\
&\le \max_b \big(T_A(\beta) + T_{u_b}(\beta) + W_{u_b}\big)\\
&\le T_A(\beta) + \max_b\big(T_{u_b}(\beta) + W_{u_b})\\
&\le\frac{1}{\lambda_{11}(\nu_{11}-1)}\log n_1 + \frac{1}{\lambda_{11}\delta_{11}\wedge\lambda_{12}\delta_{21}}\log n_1 + 2\epsilon\log n_1.
\end{align*}
The claim follows now by using assumption $n_1(\kappa) = \Theta(\kappa)$.
\end{proof}
\subsubsection{ Proof of \eqref{typical flooding time: upper bound}}

Taking maximum of flooding times $1$ and $2$, we have 
\begin{align*}
\flood(A) = \max(\flood_1(A),\flood_2(A)).
\end{align*}
Applying Theorems \ref{weighted flooding time 1:upper bound} and \ref{weighted flooding time 2:upper bound} we conclude that \whp,
\begin{align*}
\flood(A)\le \Big(\frac{1}{\lambda_{11}(\nu_{11}-1)}+\frac{1}{\lambda_{11}\delta_{11}\wedge\lambda_{12}\delta_{21}}\Big)\log \kappa +\epsilon\log \kappa.
\end{align*}

\subsection{Proof of lower bound}
In this section it will be showed that for any $\epsilon\gr 0$ \whp,
\begin{align}
\label{typical broadcast time: lower bound}
\flood(A) &\ge \Big(\frac{1}{\lambda_{11}(\nu_{11}-1)} + \frac{1}{\lambda_{11}\delta_{11}\wedge\lambda_{12}\delta_{21}} - \epsilon\Big) \log \kappa.
\end{align}

\subsection*{Notations and definitions}
The (weighted) first passage time from subset $S\subset V_1$ to $y$ in $V$ is defined by, 
\[
\tau(S,y) = \inf_{x\in S}\tau(x,y).
\]
The set of active nodes which are $t\ge 0$ time from active neighbourhood $N_1(v)$ of $v$ in $V$ is defined by
\[B'(v,t) = \{a\in V_1: \tau(N_1(v),a)\le t\}.\]
Since the weights on edges incident to $v$ in $B'(v,t)$ do not play a part, especially for passive nodes $v$ in $V_2$, we have by \cite{ADL} (replace $\dmin$ by bounded degrees in the proof of Proposition 4.13) (see also   \cite{AL}[Proposition 4.2] and \cite{DKLP}[Lemma 3.5]) the following proposition.
\begin{proposition}
\label{lower bound proofs: two balls B' with size t_n1 do not intersect}
For any $\epsilon\gr 0$ and any two distinct nodes $u$ and $v$ in $V$ with bounded degrees $\deg_1(u) = O(1)$ and $\deg_1(v)=O(1)$ \whp,
\[
B'(u,t_{n_1}) \cap B'(v,t_{n_1}) = \emptyset,
\]
where $t_{n_1} =\frac{1-\epsilon}{\lambda_{11}(\nu_{11}-1)}\log n_1$. If node $u$ (resp., $v$) is chosen uniformly at random, the same result holds without the condition $\deg_1(u) = O(1)$ (resp, $\deg_1(v) = O(1)$).
\end{proposition}

\textit{Proof of \eqref{typical broadcast time: lower bound}}.
The claim follows by applying assumptions $n_1(\kappa) = \Theta(\kappa)$ and $n_2(\kappa) = \Theta(\kappa)$ into the following proposition.

\begin{proposition}
For a uniformly chosen active node $A$ in $V_1$ and any $\epsilon\gr 0$ \whp,
\begin{align*}
\flood(A)\ge \frac{1-\epsilon}{\lambda_{11}(\nu_{11}-1)}\log n_1 + \max_{i=1,2}\Big(\frac{1-\epsilon}{\lambda_{1i}\delta_{i1}}\log n_i\Big).
\end{align*}
\end{proposition}
\begin{proof}
It is sufficient to show that each group $V_i$ has \whp a node $v_i^*$ such that
\begin{align}
\label{broadcast time: lower bound proof; the existences of bad nodes}
\tau(A,v_i^*)\ge \frac{1-\epsilon}{\lambda_{11}(\nu_{11}-1)}\log n_1 + \frac{1-\epsilon}{\lambda_{1i}\delta_{i1}} \log n_i
\end{align} 
for all $i=1,2$.\\

Denote parameters $a=\frac{1-\epsilon}{2\lambda_{11}(\nu_{11}-1)}\log n_1 $ and $b_i=\frac{1-\epsilon}{\lambda_{1i}\delta_{i1}}\log n_i$ (depending on $\kappa$), where $i=1,2$.  Let $S_{\delta_{i1}}$ be the set of type-$i$ nodes with minimum degree $\delta_{i1}$. A node in $S_{\delta_{i1}}$ is said to be bad if all of its $\delta_{i1}$-edge weights is greater than $b_i$ ($i=1,2)$. Let $C^{(i)}_v$ be the event that node $v$ in $S_{\delta_i}$ is bad. The probability of this event is
\begin{align*}
\pr(C^{(i)}_v) = \pr(\min_{u\in N_1(v)}W_{uv}\ge b_i) = n_i^{-1+\epsilon}.
\end{align*}
Let $M_i = \sum_{v}\1_{C^{(i)}_v}$ be the count of type-$i$ bad nodes in $S_{\delta_i}$. Then the expected value and the variance of $M_i$ are

\begin{align*}
\E(M_i) &=  \sum_v\pr(C^{(i)}_v) = \abs{S_{\delta_i}} n_i^{-1+\epsilon}  = (p_{i1}(\delta_i) + o(1)) n_i^{\epsilon},\\
\Var(M_i) &= \sum_{u,v}\cov(\1_{C^{(i)}_u},\1_{C^{(i)}_v}) \\
&= \sum_{u}\var(\1_{C^{(i)}_u}) + \sum_u \sum_{v\in N_1(u)}\cov(\1_{C^{(i)}_u},\1_{C^{(i)}_v})\\
&\le (\delta_i + 1) \E(M_i).
\end{align*}
By Chebyshev's inequality \whp, 
\begin{align}
\label{the numbers of bad nodes}
M_i\ge \frac{1}{2}\E(M_i) = \frac{1}{2} (p_{i1}(\delta_i) + o(1))n_i^{\epsilon}.
\end{align}  
Let $N_i$ be the number of type-$i$ bad nodes which takes at most $2a + b_i$ time to reach from $A$. Since the weights on edges incident to $A$ and $v$ do not play a part, by Proposition \ref{lower bound proofs: two balls B' with size t_n1 do not intersect}
\begin{align*}
\pr(C^{(i)}_v,B'(A,t_{n_1}) \cap B'(v,t_{n_1}) \neq \emptyset)  = o(\pr(C^{(i)}_v))
\end{align*}
and 
\begin{align*}
\E(N_i)= \sum_{v} \pr(C^{(i)}_v,B'(A,t_{n_1}) \cap B'(v,t_{n_1}) \neq \emptyset) = o(\E(M_i)).
\end{align*}
By Markov's inequality \whp $N_i\le\frac{1}{4}\E(M_i)$. Hence \whp,
\begin{align*}
M_i-N_i\ge\frac{1}{4}\E(M_i) =\frac{1}{4}(p_{i1}(\delta_i)+ o(1))n_i^{\epsilon}\ge 1,
\end{align*}
where $i=1,2$. These implies \whp there exist desired nodes $v_i^*$ satisfying \eqref{broadcast time: lower bound proof; the existences of bad nodes}. 
\end{proof}

\end{document}